\documentclass{article}
\usepackage{times,authblk}
\usepackage{amsmath, amsfonts, amssymb, amsthm, bm, stmaryrd}
\usepackage{enumitem}
\usepackage{graphicx}
\usepackage[margin=1cm,font=small]{caption}
\usepackage[gen]{eurosym}

\renewcommand{\epsilon}{\varepsilon}
\renewcommand{\rho}{\varrho}
\renewcommand{\phi}{\varphi}
\renewcommand{\theta}{\vartheta}

\newtheorem{Def}{Definition}[section]
\newenvironment{definition}{\begin{Def} \rm}{\end{Def}}
\newtheorem{lemma}[Def]{Lemma}

\newtheorem{theorem}[Def]{Theorem}

\newtheorem{remark}[Def]{Remark}

\newcommand{\EnumeratedList}{\begin{enumerate}}
\newcommand{\EndofEnumeratedList}{\end{enumerate}}
\newcommand{\Number}[1]{\item[{\rm (#1)}]\parindent0pt\parskip1ex}
\newcommand{\PointedList}{\begin{itemize}}
\newcommand{\EndofPointedList}{\end{itemize}}

\newcommand{\BeginAxioms}{\begin{itemize}[itemindent=1em]}
\newcommand{\EndAxioms}{\end{itemize}}
\newcommand{\Axiom}[1]{\item[{\rm (#1)}]\parindent0pt\parskip1ex}

\newcommand{\LAE}{$\mathsf{LAE}$}
\newcommand{\LAEC}{$\mathsf{LAEC}$}
\newcommand{\LAEPC}{$\mathsf{LAEPC}$}
\newcommand{\CPL}{$\mathsf{CPL}$}
\newcommand{\aentails}[1]{\hspace{0.05ex}>_{\!#1}\hspace{0.05ex}}
\newcommand{\modelsLAE}{\models_{\mathsf{LAE}}}
\newcommand{\provesLAE}{\vdash_{\mathsf{LAE}}}
\newcommand{\modelsLAEC}{\models_{\mathsf{LAEC}}}
\newcommand{\provesLAEC}{\vdash_{\mathsf{LAEC}}}
\newcommand{\modelsLAEPC}{\models_{\mathsf{LAEPC}}}
\newcommand{\provesLAEPC}{\vdash_{\mathsf{LAEPC}}}
\newcommand{\stronger}{\preccurlyeq}
\newcommand{\lessstrong}{\succcurlyeq}
\newcommand{\equallystrong}{\approx}
\newcommand{\equallystrongcl}[1]{\langle #1 \rangle_\approx}

\newcommand{\true}{\top}
\newcommand{\false}{\bot}
\newcommand{\tand}{\odot}
\renewcommand{\leq}{\leqslant}
\renewcommand{\geq}{\geqslant}
\newcommand{\worldleq}{\sqsubseteq}
\newcommand{\smaller}{\Diamond^{\raisebox{0.2em}{\hspace{-1.5pt}$\scriptscriptstyle \leq$}}}
\newcommand{\larger}{\Diamond^{\raisebox{0.2em}{\hspace{-1pt}$\scriptscriptstyle \geq$}}}
\newcommand{\proves}{\vdash}
\newcommand{\doesnotprove}{\nvdash}
\newcommand{\impl}{\rightarrow}
\newcommand{\eq}{\leftrightarrow}
\newcommand{\below}{\leqslant}
\renewcommand{\above}{\geqslant}

\begin{document}

\title{Logics for approximate entailment \\ in ordered universes of discourse
\thanks{Preprint of an article published by Elsevier in the {\sl International Journal of Approximate Reasoning} {\bf 71} (2016), 50-63. It is available online at: {\tt https://www.sciencedirect.com/science/ article/pii/S0888613X16300020}.}}

\author{Thomas Vetterlein${}^1$, Francesc Esteva${}^2$, Llu\'{\i}s Godo${}^2$}

\affil{\footnotesize{
${}^1$ Johannes Kepler University Linz, Altenberger Stra\ss{}e 69, 4040 Linz, Austria; \\ {\tt Thomas.Vetterlein@jku.at} \\
${}^2$ IIIA - CSIC, Campus of the UAB s/n, 08193 Bellaterra, Spain; \\ {\tt \{esteva,godo\}@iiia.csic.es}}
}

\date{}

\maketitle

\begin{abstract} 

The Logic of Approximate Entailment (\LAE) is a graded counterpart of classical propositional calculus, where conclusions that are only approximately correct can be drawn. This is achieved by equipping the underlying set of possible worlds with a similarity relation. When using this logic in applications, however, a disadvantage must be accepted; namely, in \LAE{} it is not possible to combine conclusions in a conjunctive way. In order to overcome this drawback, we propose in this paper a modification of \LAE\ where, at the semantic level, the underlying set of worlds is moreover endowed with an order structure. The chosen framework is designed in view of possible applications. 

\end{abstract}

\section{Introduction}
\label{sec:introduction}

In his seminal work on similarity-based reasoning \cite{Rus}, E.\ Ruspini proposes the interpretation of fuzzy sets in terms of (crisp) sets and fuzzy similarity relations. To this end, he builds up a framework for approximate inference that is based on the mutual similarity of the propositions involved. Following these lines, a number of approaches have dealt with similarity-based reasoning from a logical perspective \cite{DPEGG,EGGR,EGRV,GoRo}; see also \cite[Section 5.2]{LiLi}.  In particular, in the PhD thesis of R.\ Rodr\' iguez \cite{Rod}, the so-called Logic of Approximate Entailment (\LAE) is studied.

\LAE{} is a propositional logic and propositions are interpreted, as in classical logic, by subsets of a fixed set, called the set of worlds. Propositions can be logically combined like in classical propositional logic and the Boolean connectives are interpreted by the corresponding set-theoretic operations as usual. However, it is in addition assumed that the set of worlds is endowed with a fuzzy similarity relation, which associates with each pair of two worlds their degree of resemblance. The basic semantic structures are hence {\em fuzzy similarity spaces}, which consist of a set of worlds and a fuzzy similarity relation, and the core syntactic objects of \LAE\ are implications between propositions endowed with a degree. The intended meaning of a statement of the form $\alpha \aentails{c} \beta$ is that $\beta$ is an approximate consequence of $\alpha$ to the degree $c$, where $c$ is a real number between $0$ and $1$. If $c = 1$, the implication is defined to hold under the same condition as in classical propositional logic: at any world at which $\alpha$ holds, also $\beta$ must hold. If $c < 1$, however, the statement is weaker, namely, we do not require in this case that if $\alpha$ holds at a world $w$, also $\beta$ holds at $w$, we only require that there is a further world $w'$ at which $\beta$ holds and whose similarity with $w$ is at least $c$. See Figure \ref{fig:Principle} for an illustration.

\begin{figure}[h!]
\centering

\includegraphics[width=0.3\textwidth]{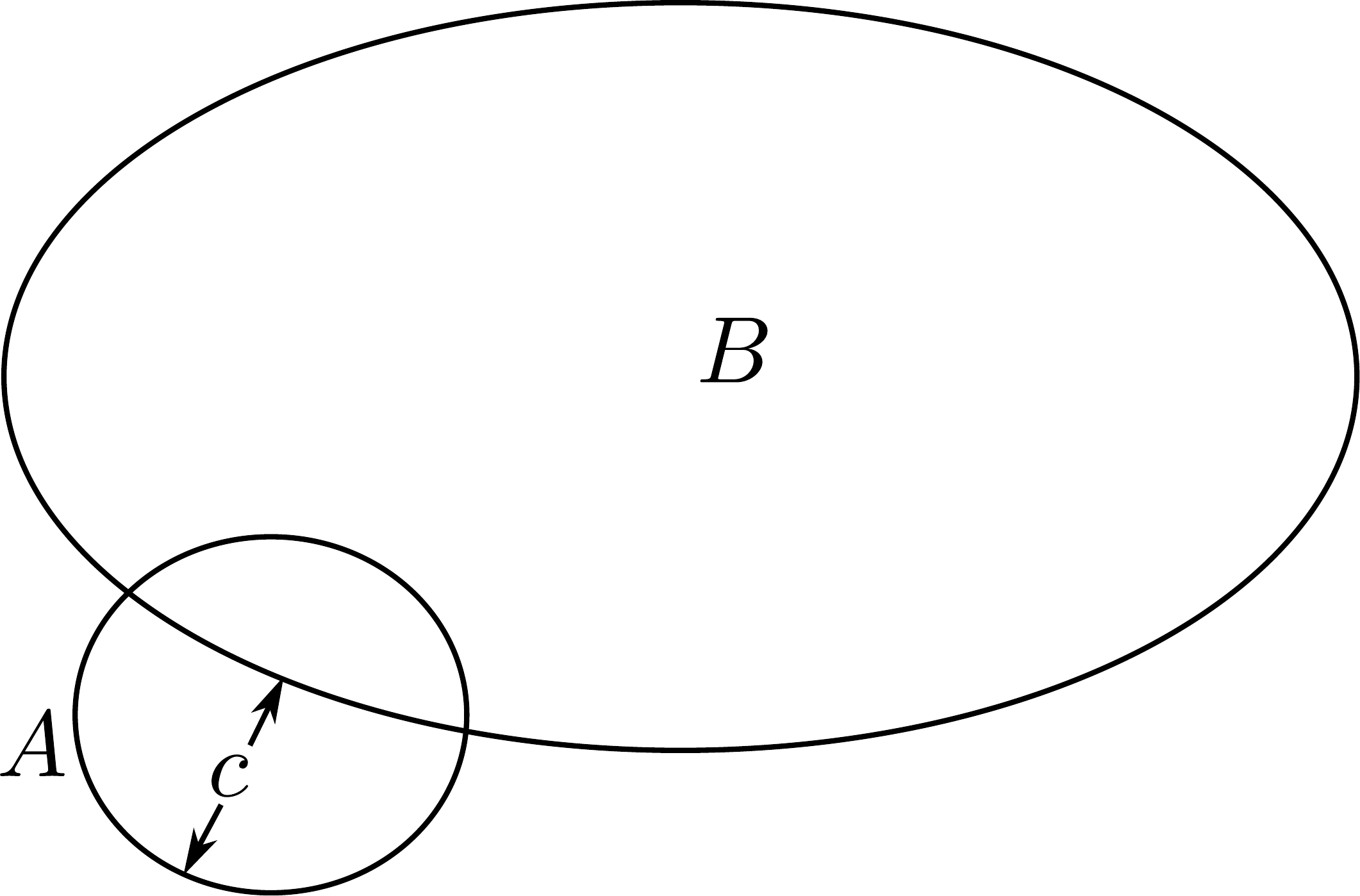}

\caption{The graded entailment in \LAE. Let $A$ and $B$ be the sets of worlds at which $\alpha$ and $\beta$ hold, respectively. Then $\alpha \aentails{c} \beta$ means that $A$ is in the $c$-neighbourhood of $B$. Note that $c$ varies between $0$ and $1$ and a smaller value of $c$ corresponds to a greater distance.}
\label{fig:Principle}
\end{figure}

Logics dealing with statements that are interpreted in metric spaces have been studied also from different points of view. Logics for spaces endowed with a metric or a more general distance function have been considered in a series of contributions by Kutz et al., see, e.g., \cite{KSSWZ,Kut}. Furthermore, logics on comparative similarity have been studied by Alenda et al., see, e.g., \cite{AlOl,AOP,AOS}. It is also worth mentioning that there are some connections with graded or fuzzy consequence relations as studied by Pavelka \cite{Pav,NPM}, Chakraborty \cite{Cha,ChDu}, and Gerla \cite{Ger} among others in the context of many-valued logics, since indeed, graded implications $\alpha \aentails{c} \beta$ capture, at a syntactic (meta-)level, the idea of $\beta$ being a consequence of $\alpha$ to the degree $c$. However, in the present context, $\alpha$ and $\beta$ are classical propositions, not many-valued ones.

The starting point for the present paper is the aforementioned logic \LAE. Although the concept underlying this logic is appealing, a disadvantage must be accepted. Deploying \LAE\ in applications is difficult for a simple reason: in \LAE{} we cannot combine conclusions in a conjunctive way. Assume that we have $\alpha \aentails{c} \beta$ and $\alpha \aentails{d} \gamma$, where $0 < c, d < 1$. Then we can not in general derive in \LAE\ a statement of the form $\alpha \aentails{e} \beta \land \gamma$ for some non-zero $e$.  This feature of \LAE\ is a straightforward consequence of the chosen semantic framework: if $\alpha$ implies that we are close to a situation in which $\beta$ holds and moreover close to a situation in which $\gamma$ holds, we cannot conclude that we are actually close to a situation in which both $\beta$ and $\gamma$ hold. In other words, for any sets of worlds $A$, $B$ and $C$, if $A$ is in the $c$-neighbourhood of $B$ as well as in the $d$-neighbourhood of $C$, we cannot make any prediction about the value $e$ such that $A$ is in the $e$-neighbourhood of $B \cap C$. Refer to Figure \ref{fig:Conjunction} for an illustration. In the extreme case, $\beta$ and $\gamma$ can even be contradictory. In such a case, there is no world at which both $\beta$ and $\gamma$ hold and $\beta \land \gamma$ will be interpreted by the empty set; but the $e$-neighbourhood of the empty set is empty for any $e$.

\begin{figure}[h!]
\centering

\includegraphics[width=0.3\textwidth]{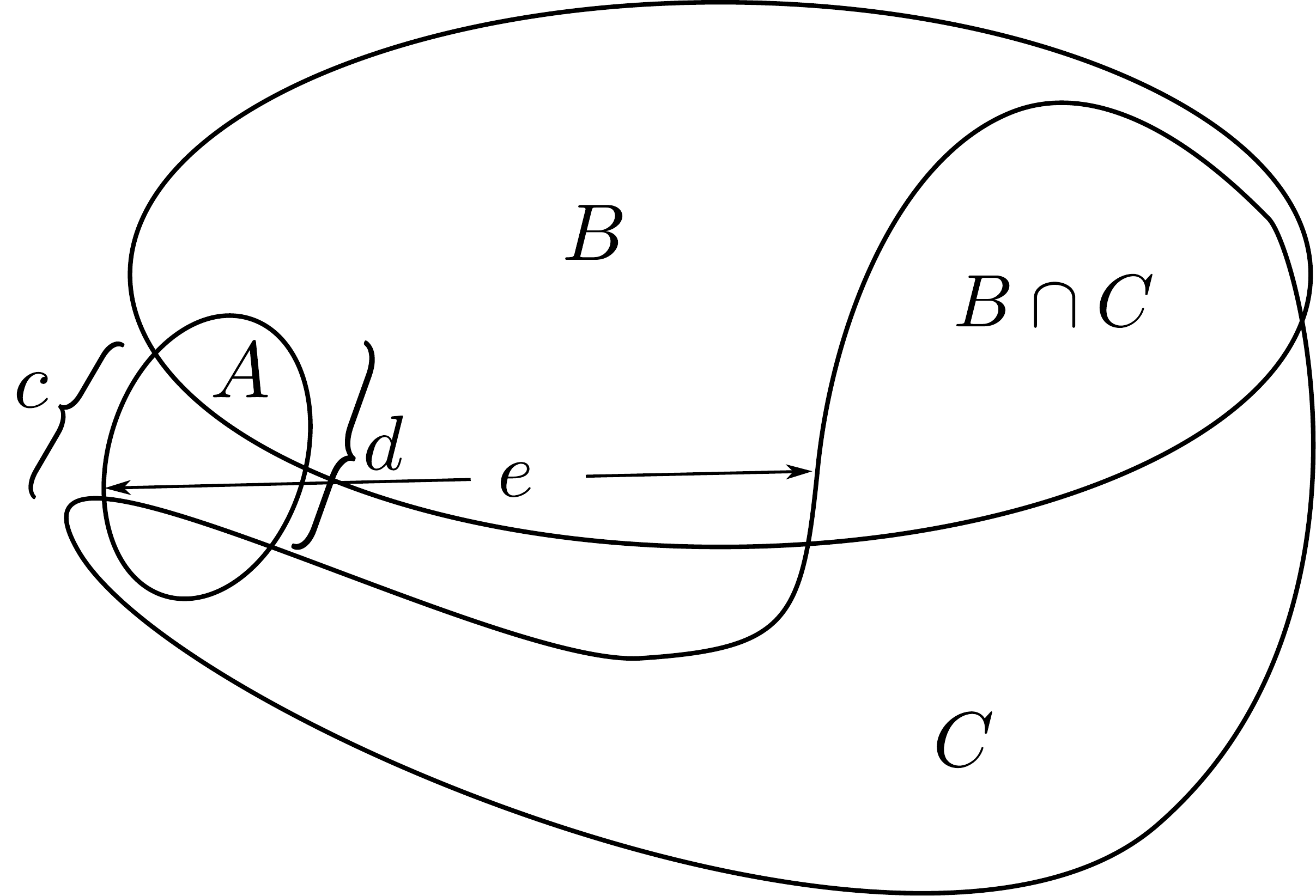}

\caption{The conjunction in \LAE. If $A$ is in the $c$-neighbourhood of $B$ as well as in the $d$-neighbourhood of $C$, we cannot make any prediction about the value $e$ such that $A$ is in the $e$-neighbourhood of $B \cap C$.}
\label{fig:Conjunction}
\end{figure}

The lack of a rule that combines conclusions in a conjunctive way may be found restrictive in applications. Let us consider the following example; let the symbols $\alpha$, $\beta$, $\gamma$ denote the following properties of a car:

\begin{tabular}{ll}
$\alpha$ & ``power(car) $=$ 110 CV'' \\
$\beta$  & ``price(car) $\geq$  20 000 \euro'' \\
$\gamma$ & ``consumption(car) $\geq$ 6 L/100km''
\end{tabular}

Assume that our domain knowledge tells us that powerful cars are expensive to some extent and at the same time they have a high consumption. These facts could be reflected by a theory containing the graded implications
\begin{equation} \label{fml:conjunction-example}
\alpha  \aentails{c} \beta, \quad \alpha  \aentails{d} \gamma,
\end{equation}
where $c$ and $d$ are some appropriate non-zero degrees. It then seems natural to be able to derive $\alpha \aentails{e} \beta \land \gamma$ for some positive degree $e$.

This situation is certainly not appropriately reflected by Figure \ref{fig:Conjunction}. The crucial difference is the independence of the properties $\beta, \gamma$ occurring in the conclusions. Price and consumption can indeed be assumed as not being interrelated. Consequently, a model can be based on a set of worlds consisting of all pairs of possible prices and possible consumption. Property $\alpha$, the power of the car, is in turn assumed to have an influence on the other two. To reflect this influence, $\alpha$ is to be identified with those pairs of a price and a consumption that are not in contradiction with it. Assuming, for instance, that a power of 110~CV implies a price range between 15~000~\euro{} and 30~000~\euro{} as well as a petrol consumption between 5~L/100km and 9~L/100km, our model would be the one indicated in Figure \ref{fig:Conjunction-independent}. Finally, a similarity between worlds can be computed as an aggregation of the similarities with regard to $\beta$ and $\gamma$, like for instance their minimum. Under these assumption, we are able to derive from (\ref{fml:conjunction-example}) the implication $\alpha \aentails{\min(a,b)} \beta \land \gamma$.

\begin{figure}[h!]
\centering

\includegraphics[width=0.3\textwidth]{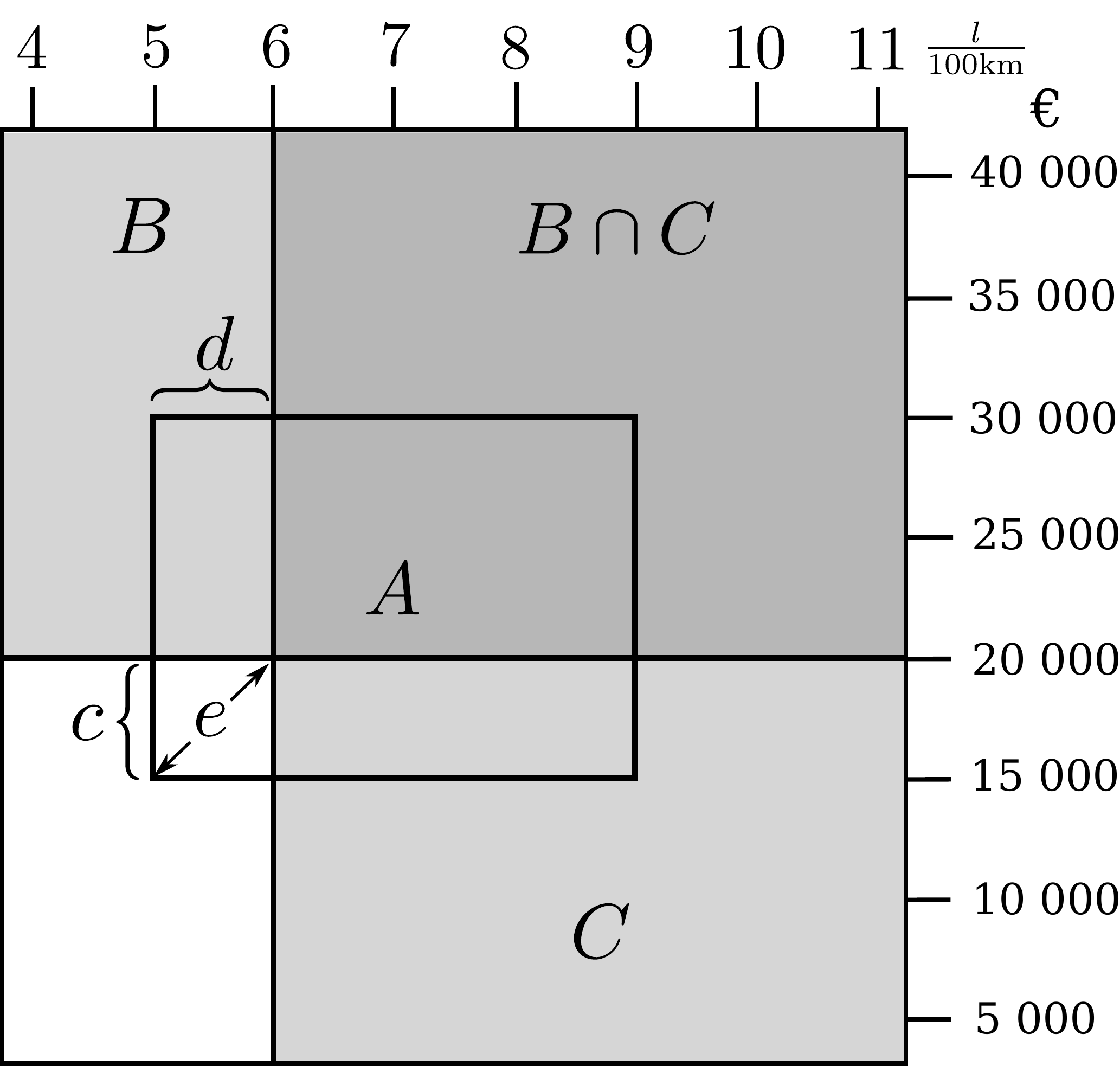}

\caption{The conjunction of independent properties, modelled by $B$ and $C$. If a set $A$ is in the $c$-neighbourhood of $B$ and in the $d$-neighbourhood of $C$, then $A$ is also in the $e$-neighbourhood of $B \cap C$, where $e$ is calculated from $c$ and $d$ according to some aggregation function.}
\label{fig:Conjunction-independent}
\end{figure}

We discuss in this paper extensions of \LAE\ that are tailored to a scenario of this kind. We shall consider two logics. As a first step, we define the logic \LAEC\ corresponding to the particular semantics where the sets of worlds in the similarity spaces are totally ordered. We show that in this case the approximate graded implication $\aentails{c}$ allows, under a natural condition, the conjunctive combination of conclusions.

In a second step, we consider the many-sorted logic \LAEPC\ whose semantics is based on similarity spaces which are Cartesian products of totally ordered ones (i.e.\  where the set of worlds is a Cartesian product of chains and the similarity over the product space is defined as the minimum of the similarities over the components). In this logic, the components in the product spaces correspond to different sorts in the logic, and  conjunctive combination of conclusions are supported whenever, roughly speaking, they are of a different sort.

The paper is organised as follows. After this introduction we review in Section \ref{sec:LAE} the basic definitions and results of the logic \LAE. In Section \ref{sec:LAEC} we present the logic  \LAEC\ for totally ordered similarity spaces, while in Section \ref{sec:LAEPC}  we introduce the more general logic \LAEPC, which can cope with products of totally ordered similarity spaces. The final Section \ref{sec:conclusion} contains some conclusions and additional remarks.

\section{The Logic of Approximate entailment}
\label{sec:LAE}

The Logic of Approximate Entailment \LAE{} is a propositional logic where propositions are interpreted by subsets of a fixed set, called the set of worlds. Propositions may be logically combined like in classical propositional logic; the Boolean connectives are interpreted by the set-theoretic operations. We will, in addition, assume that the set of worlds is endowed with a similarity relation. The core syntactic objects are implications between propositions, endowed with a degree. The intended meaning of a statement of the form $\alpha \aentails{c} \beta$ is that $\alpha$ implies $\beta$ to the degree $c$, i.e., $\beta$ is an approximate consequence of $\alpha$ to the degree $c$, where $c$ is a real number between $0$ and $1$. For $c = 1$, the implication is the classical one: if $\alpha$ holds at a world $w$, then so does $\beta$. For $c < 1$, the statement is weakened as follows: if $\alpha$ holds at a world $w$, there is a world $w'$ at which $\beta$ holds and the similarity of $w'$ and $w$ is at least $c$.

Formally, we specify \LAE{} as follows. We proceed largely in accordance with \cite{EGRV}  and we refer to this paper for further details. The original account is, however, due to \cite{Rod}; see also \cite{GoRo}. A further, recent approach to the axiomatisation of \LAE{} is contained in \cite{Vet}.

We start with a finite number $\phi_1, \ldots, \phi_N$ of variables. The number $N \geq 1$ can be chosen arbitrarily, but will be fixed. The {\it basic expressions} of \LAE{} are built up from the variables as well as the constants $\false, \true$ by means of the binary operators $\land$ and $\lor$ and the unary operator $\neg$. We denote the set of basic expressions by $\mathcal B$.

Furthermore, we choose a subset $V$ of the real unit interval containing $0$ and $1$. The elements of $V$ will be used as degrees of approximation. In a finitary setting it is moreover reasonable to assume that $V$ is finite and this is what we will do in the sequel. To express transitivity of the approximate entailment relation, we need to endow $V$ with a binary operation $\tand$ fulfilling the following conditions: $\tand$ is associative; $\tand$ is commutative; $1$ is neutral w.r.t.\ $\tand$; and $\tand$ is monotone in both arguments. In other words, $\tand$ makes $V$ into a finite, integral, commutative totally ordered monoid. In the present context, $\tand$ is also called a discrete or finite t-norm; see \cite{GoSi,MaTo}, cf.\ also \cite{DeMe}. We will assume that the pair $(V, \tand)$ is chosen arbitrarily, but fixed throughout this paper.

A {\it graded implication} of \LAE{} is a triple consisting of two basic expressions $\phi$ and $\psi$ as well as an element $c \in V$; we write
\[ \phi \aentails{c} \psi. \]
\LAE{} uses a two-level language and graded implications represent its inner level. At  the outer level, we use the classical connectives to combine graded implications into more complex expressions. That is, {\it formulas} of \LAE{} are built up from  graded implications by means of the binary operators $\land$ and $\lor$ and the unary operator $\neg$. The additional (definable) connectives $\impl$ and $\eq$ will have the usual meaning in classical propositional logic (\CPL{} for short) and to avoid brackets, these connectives will be given lowest precedence.

We next specify the semantics of \LAE.

\begin{definition} \label{def:similarity-space}
Let $W$ be a non-empty set. Let $S \colon W^2 \to V$ be such that, for any $u, v, w \in W$, (i) $S(u,v) = 1$ if and only if $u = v$; (ii) $S(u,v) = S(v,u)$; and (iii) $S(u,w) \geq S(u,v) \tand S(v,w)$. Then we call $(W,S)$ a {\it similarity space} based on $(V, \tand)$.
\end{definition}

The similarity space $(W,S)$ is intended to assume the role of a set of worlds $W$ endowed with a similarity relation $S$ that measures the resemblance of worlds. Each element $w \in W$ will give rise to a yes-no assignment of our variables and we will assume that $w$ is actually uniquely determined by this assignment. Hence our basic semantic structures will in fact be similarity spaces of finite cardinality.

For some degree $c \in V$ and a subset $A$ of a finite similarity space $(W,S)$, we put
\[ U_c(A) \;=\;
  \{ w \in W \colon \text{there is an $a \in A$ such that $S(w,a) \geq c$} \}. \]
Note that then $U_c(A) = \bigcup_{a \in A} U_c(\{a\})$ and in particular $U_c(\emptyset) = \emptyset$.

\begin{definition}
An {\it evaluation} for \LAE{} in a finite similarity space $(W, S)$ is a mapping $e \colon {\mathcal B} \to {\mathcal P}(W)$ such that (i) for any $\phi, \psi \in {\mathcal B}$, $e(\phi \land \psi) = e(\phi) \cap e(\psi)$, $\; e(\phi \lor \psi) = e(\phi) \cup e(\psi)$, $\; e(\neg\phi) = W \setminus e(\phi)$, $\; e(\false) = \emptyset$, and $e(\true) = W$ and (ii) for any distinct elements $v, w \in W$ there is a variable $\phi$ such that $e(\phi)$ contains exactly one of $v$ and $w$.


Moreover, the evaluation $e$ in $(W,S)$ is said to {\it satisfy} a graded implication $\phi \aentails{c} \psi$, written $(W, S, e) \models \phi \aentails{c} \psi$, if
\[ e(\phi) \subseteq U_c(e(\psi)); \]
the satisfaction of the remaining formulas of \LAE{} is defined in accordance with classical propositional logic.

Finally, a {\it theory} of \LAE{} is a set of formulas. We say that a theory $\mathcal T$ {\it semantically entails} a formula $\Phi$, written $\mathcal{T} \modelsLAE \Phi$, if the following holds: for any finite similarity space $(W, S)$ and any evaluation $e$ in $(W, S)$, if $(W, S, e) \models \Psi$ for all $\Psi \in {\cal T}$, then $(W, S, e) \models \Phi$. 
\end{definition}

We have defined \LAE{} on a semantic basis; we now turn to its axiomatisation. We will need an additional syntactic concept. By a {\it literal}, we mean a variable or a negated variable. A {\it maximally elementary conjunction}, or m.e.c.\ for short, is a conjunction of literals in which each variable of \LAE{} occurs exactly once. We note that these formulas are also known in the literature as {\em min-terms}.

\begin{definition} \label{def:LAE-proofs}
The following are axioms of \LAE, for any $\phi, \psi, \chi \in {\mathcal B}$ and $c, d \in V$:
\BeginAxioms
\Axiom{A1} $\phi \aentails{1} \psi$, where $\phi, \psi$ are such that $\phi \impl \psi$ is a tautology of \CPL
\Axiom{A2} $(\phi \aentails{1} \psi) \impl (\phi \land \neg\psi \aentails{1} \false)$
\Axiom{A3} $(\phi \aentails{c} \psi) \impl (\phi \aentails{d} \psi)$, where $d \leq c$
\Axiom{A4} $\lnot(\psi \aentails{1} \false) \impl (\phi \aentails{0} \psi)$
\Axiom{A5} $(\phi \aentails{c} \false) \impl (\phi \aentails{1} \false)$
\Axiom{A6} $\lnot(\delta \aentails{1} \false) \land (\delta \aentails{c} \epsilon) \;\impl\; (\epsilon \aentails{c} \delta)$, where $\delta$ and $\epsilon$ are m.e.c.'s

\Axiom{A7} $(\phi \aentails{c} \chi) \land (\psi \aentails{c} \chi) \;\impl\; (\phi \lor \psi \aentails{c} \chi)$
\Axiom{A8} $(\epsilon \aentails{c} \phi \lor \psi) \;\impl\; (\epsilon \aentails{c} \phi) \lor (\epsilon \aentails{c} \psi)$, where $\epsilon$ is a m.e.c.
\Axiom{A9} $(\phi \aentails{c} \psi) \land (\psi \aentails{d} \chi) \impl (\phi \aentails{c \tand d} \chi)$
\Axiom{A10} $\lnot(\top \aentails{1} \false)$
\Axiom{A11} For any tautology of \CPL, the formula resulting from a uniform replacement of the variables by graded implications.
\EndAxioms
Finally, modus ponens is the only rule of \LAE{}: for any formulas $\Phi, \Psi$
\BeginAxioms
\Axiom{MP} $\displaystyle{\frac{\Phi \quad \Phi \impl \Psi}{\Psi}}$
\EndAxioms
A proof of a formula $\Phi$ from a theory $\mathcal T$ is defined as usual. If it exists, we write $\mathcal{T} \vdash_{\mathrm{LAE}} \Phi$.

A theory $\mathcal T$ is called {\it consistent} if $\mathcal T$ does not prove $\true \aentails{1} \false$.
\end{definition}

Let us shortly comment on the axioms (A1)--(A10) of \LAE. Basic expressions are interpreted by subsets of a similarity space and $\aentails{1}$ corresponds to the subsethood relation; thus the meaning of (A1) and (A2) is clear. Recall furthermore that, in general, $\aentails{c}$ expresses that a set is contained in the $c$-neighbourhood of another set. A $c$-neighbourhood is contained in a $d$-neighbourhood if $d \leq c$; hence also (A3) is justified. (A4) says that the $0$-neighbourhood of a non-empty set is the whole set of worlds. (A5) in turn expresses that the $c$-neighbourhood of the empty set is the empty set for whatever value $c$.

Next note that, by part (ii) of the definition of an evaluation, the set interpreting a m.e.c.\ contains at most one element. Accordingly, (A6) means that if a singleton is in the $c$-neighbourhood of another singleton, then also the converse holds. (A7) expresses that if two sets are in the $c$-neighbourhood of a further set, then so is its union. And (A8) says that if an element is in the $c$-neighbourhood of the union of sets, then it is already in the $c$-neighbourhood of one of these sets.

(A9) is the transitivity for the neighbourhood relations. Finally, (A10) asserts that the set of worlds is not empty. Note that, due to axiom (A10), consistency in \LAE{} is equivalent to consistency in the sense of classical propositional logic.

Let us moreover remark that \LAE{} satisfies the rule of substitution of classical equivalents. Indeed, if $\alpha \leftrightarrow \alpha'$ and $\beta \leftrightarrow \beta'$ are tautologies of \CPL, then $(\alpha \aentails{c} \beta) \leftrightarrow (\alpha' \aentails{c} \beta')$. For, we conclude by (A1) that $\alpha' \aentails{1} \alpha$ and $\beta \aentails{1} \beta'$ and hence, by (A9), $(\alpha \aentails{c} \beta) \rightarrow (\alpha' \aentails{c} \beta')$. Similarly we argue for the converse direction.

We have the following soundness and completeness theorem for \LAE. We include a proof that is just detailed enough to serve as a reference in the subsequent sections; for full details, we refer to \cite{EGRV}.

\begin{theorem} \label{thm:completeness-LAE}
Let $\mathcal T$ be a theory and $\Phi$ be a formula of \LAE. Then ${\mathcal T} \provesLAE \Phi$ if and only if ${\mathcal T} \modelsLAE \Phi$.
\end{theorem}

\begin{proof}
 The ``only if'' part follows from the soundness of the axioms (A1)--(A10), which is easily checked; cf.\ the explanations after Definition \ref{def:LAE-proofs}.

To show the ``if'' part, assume that $\mathcal T$ does not prove $\Phi$. Extending $\mathcal T$ if necessary, we can w.l.o.g.\ assume that $\mathcal T$ is complete, that is, ${\mathcal T}  \provesLAE \Psi$ or ${\mathcal T} \provesLAE \lnot\Psi$ for each formula $\Psi$.

For $\phi, \psi \in {\mathcal B}$, let $\phi \stronger \psi$ if  ${\mathcal T} \proves \phi \aentails{1} \psi$. Let $\equallystrong$ be the symmetrisation of $\stronger$, that is, let $\phi \equallystrong \psi$ if $\phi \stronger \psi$ and $\psi \stronger \phi$. We conclude from (A1) and (A2) that $\equallystrong$ is an equivalence relation and in fact a congruence with respect to $\land$, $\lor$, and $\lnot$. We denote the $\equallystrong$-class of some $\phi \in {\mathcal B}$ by $\equallystrongcl{\phi}$. Let $\mathcal L = \{ \equallystrongcl{\phi} \colon \phi \in \mathcal B \}$; then  $\mathcal L$, endowed with the induced operations $\land$, $\lor$, $\lnot$ as well as the constants $\equallystrongcl{\false}$, $\equallystrongcl{\true}$, is a Boolean algebra.

As there are only finitely many variables, $\mathcal B$ and consequently also $\mathcal L$ are finite. Furthermore, any $\phi \in \mathcal B$ is the supremum of m.e.c.s and hence the atoms of $\mathcal L$ are $\equallystrongcl{\epsilon}$, where $\epsilon$ is a m.e.c.\ such that ${\mathcal T} \doesnotprove \epsilon \aentails{1} \false$. Let $W$ be the set of atoms of $\mathcal L$, and for $\phi \in {\mathcal B}$, let
\begin{equation} \label{fml:evaluation}
e(\phi) \;=\;
   \{ \equallystrongcl{\epsilon} \in W \colon \epsilon \stronger \phi \}.
\end{equation}
Then $e \colon {\mathcal B} \to {\mathcal P}(W)$ is an evaluation for \LAE.

For $\equallystrongcl{\delta}, \equallystrongcl{\epsilon} \in W$, we define
\begin{equation} \label{fml:derived-similarity}
S(\equallystrongcl{\delta}, \equallystrongcl{\epsilon}) \;=\; \max\; \{ c \in V \colon {\mathcal T} \proves \delta \aentails{c} \epsilon \};
\end{equation}
note that, due to the finiteness of $V$, the maximum always exists. By (A1), (A6) and (A9), respectively, $S$ is reflexive, symmetric and $\odot$-transitive. Furthermore, if $S(\equallystrongcl{\delta}, \equallystrongcl{\epsilon}) = 1$, then ${\mathcal T} \proves \delta \aentails{1} \epsilon$; this means that ${\mathcal T} \proves \epsilon  \aentails{1} \delta$ holds as well, hence $\delta \equallystrong \epsilon$. We conclude that $S$ is a similarity relation. Moreover, we have, for any $\phi, \psi \in {\mathcal B}$,
\begin{equation} \label{fml:LAE-proof}
{\mathcal T} \proves \phi \aentails{c} \psi \quad\text{iff}\quad
e(\phi) \subseteq U_c(e(\psi)).
\end{equation}
Indeed, the case that $\phi \equallystrong \false$ or $\psi \equallystrong \false$ or $c = 0$ is covered by (A1), (A4), and (A5). Otherwise, we have by (A1), (A3), (A7), (A8), and the completeness of $\mathcal T$ that ${\mathcal T} \proves \phi \aentails{c} \psi$ iff ${\mathcal T} \proves \bigvee_{\delta \stronger \phi} \delta \aentails{c} \bigvee_{\epsilon \stronger \psi} \epsilon$, where $\delta$ and $\epsilon$ are meant to refer to m.e.c.'s. Now the latter holds iff for any $\delta \stronger \phi$ there is an $\epsilon \stronger \psi$ such that ${\mathcal T} \proves \delta \aentails{c} \epsilon$, and this holds iff for any $\equallystrongcl{\delta} \in e(\phi)$ there is an $\equallystrongcl{\epsilon} \in e(\psi)$ such that $S(\equallystrongcl{\delta}, \equallystrongcl{\epsilon}) \geq c$, and finally this holds iff $e(\phi) \subseteq U_c(e(\psi))$. 

It follows that $e$ satisfies all elements of $\mathcal T$, but not $\Phi$. That is, $\mathcal T$ does not semantically entail $\Phi$.
\end{proof}

\section{Approximate Entailment on a Chain}
\label{sec:LAEC}

The Logic of Approximate Entailment \LAE, which we have defined in the previous section, might be appealing because of its transparency and simplicity. However, we should admit that the practical usability of \LAE{} is limited. Roughly speaking, we may observe that \LAE{} is well-behaved as regards the logical disjunction, but poorly behaved as regards the logical conjunction. In \LAE, like in classical propositional logic, we can indeed derive from $\phi \aentails{c} \chi$ and $\psi \aentails{c} \chi$ that $\phi \lor \psi \aentails{c} \chi$, and vice versa. In contrast, assume that we have $\phi \aentails{c} \chi$ and $\phi \aentails{c} \psi$, a probably even more common situation. In general, there is no way to derive in \LAE{} from these statements alone that $\phi \aentails{d} \chi \land \psi$ such that $d > 0$. The only exception is the case $c = 1$, which allows classical reasoning.

To respond to this weakness, we consider in this subsection a more special framework. We assume that we proceed in accordance with typical applications. Indeed, configurations are often described by parameters and distinctions are often made by the reference to a totally ordered structure. In this paper, we consider two variants of \LAE{} that are based on exactly this assumption.

In our first step, which is the topic of this section, we will assume that our set of worlds is equipped with a total order, so they form a chain. We will extend our language so as to be able to refer to the total order. In a subsequent step, discussed in Section \ref{sec:LAEPC}, we will go further and work with a direct product of chains.

In order to be able to refer to a total order, we will use two modal operators, denoted by $\smaller$ and $\larger$. Disregarding the similarity relation, our new calculus is closely related to the modal logic S4.3; see, e.g., \cite{HuCr}. Our notation is chosen accordingly; the particular symbol $\smaller$ is borrowed from \cite{vBGR}. S4.3 can be viewed as the logic of total preorders, where in a Kripke model a proposition $\Diamond\phi$ is interpreted by the set consisting of the worlds at which $\phi$ holds as well as those that are below them according to the given preorder. Replacing ``preorder'' with ``order'', this is exactly the way we will interpret $\smaller$. In addition, we add the operator $\larger$, which will be interpreted in a dual way.

Let us now formally specify the Logic for Approximate Entailment on a Chain, or \LAEC{} for short. We fix again a finite number of variables $\phi_1, \ldots, \phi_N$. The basic expressions of \LAEC{} are built up from the variables and the constants $\false, \true$ by means of the Boolean connectives and the two unary operators $\smaller$ and $\larger$. To denote the set of basic expressions of \LAEC, we use again the symbol $\mathcal B$. A basic expression of the form $\larger\phi$ or $\smaller\phi$ for some $\phi \in {\mathcal B}$ will be called a {\em diamond expression}. Similarly to the case of \LAE, graded implications of \LAEC{} are expressions of the form $\varphi \aentails{c} \psi$, where $\varphi, \psi \in \mathcal B$, and formulas of \LAEC{} are Boolean combinations of graded implications.


On the semantic side, our basic models will include a total order.

\begin{definition} \label{def:totally-ordered-similarity-space}
We call a triple $(W, S, \below)$ a {\it totally ordered similarity space} if: (i) $(W,S)$ is a finite similarity space based on $(V, \tand)$, and (ii) $\below$ is a total order on $W$ such that, for $u, v, w \in W$, $u \below v \below w$ implies $\min(S(u,v), S(v,w)) \geq S(u,w)$.
\end{definition}

Note that condition (ii) of Definition \ref{def:totally-ordered-similarity-space} requires  the similarity relation $S$ in a totally ordered similarity space $(W, S, \leq)$ to be compatible with the underlying ordering $\leq$ in a natural sense: proceeding from some element $u$ along the chain upwards or downwards, the similarity with $u$ becomes successively smaller.

In order to define the interpretation of the two new unary operators of \LAEC, let us make the following definitions. Let $(W, S, \below)$ be a totally ordered similarity space and let $A \subseteq W$ be a subset of worlds; then we put:
\begin{equation} \label{fml:diamond-on-similarity-space}
\begin{split}
\smaller A \;=\; & \{ v \in W \colon
 \text{there is a $w \in A$ such that $v \below w$} \}, \\
\larger A \;=\; & \{ v \in W \colon
 \text{there is a $w \in A$ such that $v \above w$} \}.
\end{split}
\end{equation}

In a finite chain $(W; \below)$,  given two elements $u, v \in W$ such that $u \leq v$, we call a set of the form $[u,v] = \{ w \in W \colon u \leq w \leq v \}$ an {\it interval} of $W$. If $u$ is the bottom element of $W$, we also write $(-\infty,v]$; if $v$ is the top element, we also write $[u,\infty)$. Note that, for any non-empty $A \subseteq W$, we have
\begin{align*}
\smaller A \;=\; & (-\infty, \max A], \\
\larger A \;=\; & [\min A, \infty).
\end{align*}
Therefore, $\smaller A$ and $\larger A$ are intervals in $W$. Trivially, the intersection of intervals is either empty or an interval as well. Note that, in particular, $\smaller A \cap \larger A$ is the smallest interval containing $A$.

The following lemma compiles some obvious properties of the operators (\ref{fml:diamond-on-similarity-space}).

\begin{lemma} \label{lem:diamond-on-similarity-space}
Let $(W, S, \below)$ be a totally ordered similarity space. Then we have, for any $A, B \subseteq W$ and $c \in V$:
\EnumeratedList
\Number{i} $A \subseteq \smaller A$.
\Number{ii} $\smaller{\smaller A} = \smaller A$.
\Number{iii} $\smaller\emptyset = \emptyset$.
\Number{iv} At least one of $\smaller A \subseteq \smaller B$ or $\smaller B \subseteq \smaller A$ holds.
\Number{v} If $A \subseteq U_c(B)$, then $\smaller A \subseteq U_c \smaller B$.

In particular, $A \subseteq B$ implies $\smaller A \subseteq \smaller B$.
\EndofEnumeratedList
In addition, each statement {\rm (i)}--{\rm (v)} still holds when we replace all symbols ``$\smaller$'' by ``$\larger$''.

Moreover, we have for all $A, B \subseteq W$:
\EnumeratedList
\Number{vi} For any $w \in W$, $\smaller{ \{w\} } \cap \larger{ \{w\} } = \{w\}$.
\Number{vii} $A \cap \smaller B = \emptyset$ implies $\larger A \cap \smaller B = \emptyset$. Similarly, $A \cap \larger B = \emptyset$ implies $\smaller A \cap \larger B = \emptyset$.
\EndofEnumeratedList
\end{lemma}

One more statement concerns our very motivation to define the logic \LAEC: the behaviour of the approximate implication $\aentails{c}$ with respect to conjunction.

\begin{lemma} \label{lem:intersection-of-intervals}
Let $(W, S, \below)$ be a totally ordered similarity space. For intervals $A$ and $B$ of $W$ with a non-empty intersection and for any $c \in V$, we have
\[ U_c(A \cap B) \;=\; U_c(A) \cap U_c(B). \]
\end{lemma}

Finally, we can define the evaluations in \LAEC. 

\begin{definition} \label{def:semantics-LAEC}
An {\it evaluation} $e$ for \LAEC{} in a totally ordered similarity space $(W, S,$ $\below)$  is defined as in case of \LAE; in addition, we require, for any basic expression $\phi$, that $e(\smaller\phi) = \smaller{e(\phi)}$ and $e(\larger\phi) = \larger{e(\phi)}$.\footnote{Note that we use the same symbols $\smaller$ and $\larger$ for both the syntactic and semantic operators, but it will be clear from the context when we refer to one or the other.}

Moreover, the notions of {\it satisfaction} and {\it semantic entailment} are defined analogously to the case of \LAE. We write $(W, S, \below, e) \models \Phi$ to denote that an evaluation $e$ in $(W, S, \below)$ satisfies an \LAEC\ formula $\Phi$, and $\mathcal{T} \modelsLAEC \Phi$ to denote that a theory $\cal T$ semantically entails $\Phi$ in \LAEC. 
\end{definition}

We now turn to the axiomatisation of \LAEC. Note that we may indeed observe a close relationship of \LAEC{} with the modal logic S4.3. Recall that S4.3 is the extension of the basic modal logic $K$ by the axioms $T$, $4$ and $H$. Our axioms (A12), (A13) are analogues of the modal axioms $T$ and $4$, respectively, and (A15) is related to axiom $H$, as we will see later.

\begin{definition} \label{def:LAEC-proofs}
The axioms of \LAEC{} are the following ones, for any basic expressions of \LAEC{} $\phi, \psi, \chi \in {\mathcal B}$ and any $c \in V$: 

\BeginAxioms
\Axiom{A1'} $\phi \aentails{1} \psi$, where, for formulas $\phi', \psi'$ of \CPL, $\phi' \impl \psi'$ is a tautology of \CPL{} and $\phi, \psi$ arise, respectively, from $\phi'$ and $\psi'$ by a uniform replacement of the variables occurring in $\phi'$ or $\psi'$ by basic expressions of \LAEC;
\EndAxioms

axioms (A2)--(A11);

\vspace{1ex}

\begin{minipage}{0.49\textwidth}
\BeginAxioms
\Axiom{A12a} $\phi \aentails{1} \smaller\phi$
\Axiom{A13a} $\smaller{\smaller\phi} \aentails{1} \smaller\phi$
\Axiom{A14a} $\smaller\false \aentails{1} \false$
\Axiom{A15a} $(\smaller\phi \aentails{1} \smaller\psi) \lor (\smaller\psi \aentails{1} \smaller\phi)$
\EndAxioms
\end{minipage}
\hfill
\begin{minipage}{0.49\textwidth}
\BeginAxioms
\Axiom{A12b} $\phi \aentails{1} \larger\phi$
\Axiom{A13b} $\larger{\larger\phi} \aentails{1} \larger\phi$
\Axiom{A14b} $\larger\false \aentails{1} \false$
\Axiom{A15b} $(\larger\phi \aentails{1} \larger\psi) \lor (\larger\psi \aentails{1} \larger\phi)$
\EndAxioms
\end{minipage}

\vspace{1ex}

\BeginAxioms
\Axiom{A16} $\smaller\epsilon \land \larger\epsilon \aentails{1} \epsilon$, where $\epsilon$ is a m.e.c.
\EndAxioms

\vspace{1ex}

\begin{minipage}{0.49\textwidth}
\BeginAxioms
\Axiom{A17a} $(\phi \aentails{c} \psi) \impl (\smaller\phi \aentails{c} \smaller\psi)$
\Axiom{A18a} $(\phi \land \smaller\psi \aentails{1} \false) \\  \hspace*{1.7cm}\impl (\larger\phi \land \smaller\psi \aentails{1} \false)$
\EndAxioms
\end{minipage}
\hfill
\begin{minipage}{0.49\textwidth}
\BeginAxioms
\Axiom{A17b} $(\phi \aentails{c} \psi) \impl (\larger\phi \aentails{c} \larger\psi)$
\Axiom{A18b} $(\phi \land \larger\psi \aentails{1} \false) \\  \hspace*{1.7cm}\impl (\smaller\phi \land \larger\psi \aentails{1} \false)$
\EndAxioms
\end{minipage}

\vspace{1ex}

\BeginAxioms
\Axiom{A19} $\lnot(\rho \land \sigma \aentails{1} \false) \land (\phi \aentails{c} \rho) \land (\phi \aentails{c} \sigma) \impl (\phi \aentails{c} \rho \land \sigma)$, \\ where $\rho, \sigma$ are conjunctions of diamond expressions.
\EndAxioms
Moreover, the only rule of \LAEC{} is (MP).
\end{definition}

Note that the last axiom $(\textrm{A19})$ literally captures the above mentioned conjunctive property of the $\aentails{c}$ relations. Indeed, a conjunction of diamond expressions is interpreted by an interval. Hence (A19) says that if two non-contradictory properties $\rho$ and $\sigma$ that correspond to intervals follow from the same premise $\varphi$ to the degree $c$, then their conjunction $\rho \land \sigma$ follows from $\varphi$ to the same degree $c$. Due to Lemma \ref{lem:intersection-of-intervals}, (A19) is obviously sound.

\begin{lemma} \label{lem:LAEC}
In \LAEC, we can derive, for any basic expressions $\phi$ and $\psi$,
\begin{equation} \label{fml:disjunction-rule}
\smaller(\phi \lor \psi) \aentails{1} \smaller\phi \lor \smaller\psi
\quad\text{\rm and} \quad
\smaller\phi \lor \smaller\psi \aentails{1} \smaller(\phi \lor \psi),
\end{equation}
and similarly
\begin{equation*}
\larger(\phi \lor \psi) \aentails{1} \larger\phi \lor \larger\psi.
\quad\text{\rm and} \quad
\larger\phi \lor \larger\psi \aentails{1} \larger(\phi \lor \psi).
\end{equation*}
\end{lemma}

\begin{proof}
We only show the first half of the lemma; the second one is seen analogously.

From $\phi \aentails{1} \smaller\phi$ and $\psi \aentails{1} \smaller\psi$, we get $\phi \lor \psi \aentails{1} \smaller\phi \lor \smaller\psi$. By (A15a), we have $(\smaller\phi \aentails{1} \smaller\psi) \lor (\smaller\psi \aentails{1} \smaller\phi)$. Now we reason by cases. From $\smaller\phi \aentails{1} \smaller\psi$, by $(A12a)$ we conclude   $\phi \lor \psi \aentails{1} \smaller\psi$ and by (A17a) and (A13a) $\smaller{(\phi \lor \psi)} \aentails{1} \smaller\psi$ and hence $\smaller{(\phi \lor \psi)} \aentails{1} \smaller\phi \lor \smaller\psi$. Similarly, from $\smaller\psi \aentails{1} \smaller\phi$ we conclude that $\smaller{(\phi \lor \psi)} \aentails{1} \smaller\phi \lor \smaller\psi$ as well. Therefore, the first part of (\ref{fml:disjunction-rule}) follows.

Since $\varphi \to \varphi \lor \psi$ is a tautology of CPL, we derive $\varphi  \aentails{1} \varphi \lor \psi$ by (A1), and by (A17a) it follows $\smaller\varphi  \aentails{1} \smaller(\varphi \lor \psi)$. We argue similarly to derive $\smaller\psi  \aentails{1} \smaller(\varphi \lor \psi)$ and thus by propositional reasoning we have $\smaller \varphi \lor \smaller\psi  \aentails{1} \smaller(\varphi \lor \psi)$. Therefore, also the second part of (\ref{fml:disjunction-rule}) holds.
\end{proof}

\begin{theorem} \label{thm:completeness-LAEC}
Let $\mathcal T$ be a theory and $\Phi$ be a formula of \LAEC. Then ${\mathcal T} \provesLAEC \Phi$ if and only if ${\mathcal T} \modelsLAEC \Phi$.
\end{theorem}

\begin{proof}
As regards the ``only if'' part, the soundness of (A1)--(A10) follows from Theorem \ref{thm:completeness-LAE}. The soundness of (A12)--(A18) follows from Lemma \ref{lem:diamond-on-similarity-space}. The soundness of (A19) holds by Lemma \ref{lem:intersection-of-intervals}. 

To see the ``if'' part, assume that $\mathcal T$ does not prove $\Phi$. Extending $\mathcal T$ if necessary, we can assume that the theory $\mathcal T$ is complete.

We write $\phi \stronger \psi$ for ${\mathcal T} \proves \phi \aentails{1} \psi$, we let $\equallystrong$ be the symmetrisation of $\stronger$, and we proceed like in the proof of Theorem \ref{thm:completeness-LAE} to construct the Boolean algebra $\mathcal L$ of $\equallystrong$-classes. Furthermore, by (A17) also $\smaller$ and $\larger$ are compatible with $\equallystrong$; hence we can define $\smaller\equallystrongcl{\phi} = \equallystrongcl{\smaller\phi}$ and $\larger\equallystrongcl{\phi} = \equallystrongcl{\larger\phi}$ for any $\equallystrongcl{\phi} \in {\mathcal L}$.

Let $W$ again be the set consisting of the classes $\equallystrongcl{\epsilon}$, where $\epsilon$ is a m.e.c.\ such that ${\mathcal T} \doesnotprove \epsilon \aentails{1} \false$. Then the infimum of two distinct elements of $W$ is $\equallystrongcl{\false}$ and the supremum of all elements of $W$ is $\equallystrongcl{\true}$.

We claim that, for any m.e.c.\ $\equallystrongcl{\epsilon} \in W$, $\equallystrongcl{\smaller\epsilon}$ is a supremum of elements of $W$. Indeed, let $\delta$ be a further m.e.c.; the claim will follow from the fact that either $\delta \stronger \smaller\epsilon$ or $\delta \land \smaller\epsilon \equallystrong \false$. If $\delta$ and $\epsilon$ coincide, the first case applies because of (A12). Let $\delta$ and $\epsilon$ be distinct. By (A15), one of the following two possibilities applies:

{\it Case 1.} $\larger\delta \stronger \larger\epsilon$. Then $\delta \land \smaller\epsilon \equallystrong \delta \land \larger\delta \land \smaller\epsilon \stronger \delta \land \larger\epsilon \land \smaller\epsilon \equallystrong \delta \land \epsilon \equallystrong \false$ by (A12) and (A16).

{\it Case 2.} $\larger\epsilon \stronger \larger\delta$. In this case, if we furthermore have that $\smaller\epsilon \stronger \smaller\delta$, it follows $\epsilon \equallystrong \smaller\epsilon \land \larger\epsilon \stronger \smaller\delta \land \larger\delta \equallystrong \delta$ by (A16), in contradiction to our assumption that $\delta$ and $\epsilon$ are distinct. Hence, by (A15), we have $\smaller\delta \stronger \smaller\epsilon$ and thus, by (A12), $\delta \stronger \smaller\epsilon$.

By Lemma \ref{lem:LAEC}, it further follows that, for any basic expression $\phi$ in which $\smaller$ and $\larger$ does not occur, $\equallystrongcl{\smaller\phi}$ is a supremum of elements of $W$. We may argue similarly to conclude that the same applies to $\equallystrongcl{\larger\phi}$. Hence $W$ is the set of atoms of $\mathcal L$.

We next define $e \colon {\mathcal B} \to {\mathcal P}(W)$ by (\ref{fml:evaluation}). Then $e$ preserves the Boolean operations and constants. Furthermore, we define $S \colon W^2 \to V$ as in (\ref{fml:derived-similarity}), that is, $S(\equallystrongcl{\delta}, \equallystrongcl{\epsilon}) = \max\{ c \in V \mid \mathcal{T} \vdash \delta  \aentails{c}  \epsilon \}$. We conclude as in the proof of Theorem \ref{thm:completeness-LAE} that $(W, S)$ is a finite similarity space such that (\ref{fml:LAE-proof}) holds.

For $\equallystrongcl{\delta}, \equallystrongcl{\epsilon} \in W$, let
\[ \equallystrongcl{\delta} \below \equallystrongcl{\epsilon} \quad\text{if}\quad
   \smaller\delta \stronger \smaller\epsilon. \]
Note that, by (A15) and (A16), $\equallystrongcl{\delta} \below \equallystrongcl{\epsilon}$ iff $\larger\epsilon \stronger \larger\delta$. We claim that $\below$ totally orders $W$. Indeed, reflexivity is clear; antisymmetry holds by (A16); and transitivity is evident as well. Finally, the linearity of $\below$ holds by (A15).

Let $\equallystrongcl{\delta}, \equallystrongcl{\epsilon}, \equallystrongcl{\zeta} \in W$ such that $\equallystrongcl{\delta} \below \equallystrongcl{\epsilon} \below \equallystrongcl{\zeta}$ and assume ${\mathcal T} \proves \delta \aentails{c} \zeta$. Then we have ${\mathcal T} \proves \larger\epsilon \aentails{1} \larger\delta$ and, by (A17), ${\mathcal T} \proves \larger\delta \aentails{c} \larger\zeta$; hence ${\mathcal T} \proves \larger\epsilon \aentails{c} \larger\zeta$ as well. We also have that ${\mathcal T}$ proves $\smaller\epsilon \aentails{1} \smaller\zeta$, and hence $\smaller\epsilon \land \larger\epsilon \aentails{c} \larger\zeta$ as well as $\smaller\epsilon \land \larger\epsilon \aentails{c} \smaller\zeta$. Furthermore, by (A16), $\smaller\zeta \land \larger\zeta \equallystrong \zeta$ and by assumption ${\mathcal T} \doesnotprove \zeta \aentails{1} \false$; hence, by the completeness of $\mathcal T$, ${\mathcal T} \proves \lnot(\smaller\zeta \land \larger\zeta \aentails{1} \false)$. We conclude by (A19) that ${\mathcal T} \proves \smaller\epsilon \land \larger\epsilon \aentails{c} \smaller\zeta \land \larger\zeta$ and, by (A16), ${\mathcal T} \proves \epsilon \aentails{c} \zeta$. It follows that $S(\equallystrongcl{\epsilon}, \equallystrongcl{\zeta}) \geq S(\equallystrongcl{\delta}, \equallystrongcl{\zeta})$.

Similarly, we proceed to derive also $S(\equallystrongcl{\delta}, \equallystrongcl{\epsilon}) \geq S(\equallystrongcl{\delta}, \equallystrongcl{\zeta})$. Thus, we have shown that $(W, S, \below)$ is a totally ordered similarity space.

It remains to show that $e$ preserves the modal operations; it will then follow that $e$ is an evaluation for \LAEC. For $\equallystrongcl{\delta}, \equallystrongcl{\epsilon} \in W$, we have $\delta \stronger \smaller\epsilon$ if and only if $\smaller\delta \stronger \smaller\epsilon$ if and only if $\equallystrongcl{\delta} \below \equallystrongcl{\epsilon}$. That is, $\smaller\epsilon$ is the supremum of all $\delta$ such that $\equallystrongcl{\delta} \below \equallystrongcl{\epsilon}$. It follows that, for any $\equallystrongcl{\phi} \in {\mathcal L}$,
\begin{align*}
\smaller\equallystrongcl{\phi} & \;=\; \equallystrongcl{\smaller\phi} \\
& \;=\; \bigvee \{ \equallystrongcl{\smaller\epsilon} \colon 
        \equallystrongcl{\epsilon} \in W \text{ such that } \epsilon \stronger \phi \} \\
& \;=\; \bigvee \{ \equallystrongcl{\delta} \in W \colon
         \equallystrongcl{\delta} \below \equallystrongcl{\epsilon}
         \text{ for some $\equallystrongcl{\epsilon} \in W$ such that $\epsilon \stronger \phi$} \},
\end{align*}
and a similar statement holds for $\larger\phi$.

We summarise that $e$ is an evaluation, which, by (\ref{fml:LAE-proof}), satisfies the elements of $\mathcal T$ but not $\Phi$.
\end{proof}

\begin{remark}
We have enlarged the logic \LAE{} by a pair of two modal operators, $\smaller$ and $\larger$. We may say that they are order-theoretic duals of each other. We note that we could build a logic like \LAEC{} also on the basis of one of these operations, say $\larger$, alone. Indeed, formulas of the form $\larger\alpha \land \lnot\larger\beta$ are suitable to represent the intervals of the totally ordered set of worlds.

This approach, however, would differ from the present one to a larger extent than one might expect. Eliminating $\smaller$ would be easy if we could express $\smaller\alpha$ by a formula containing $\larger$ only; there seems to be no straightforward way of doing so.
\end{remark}

\begin{remark}
The modal logic {\rm S4.3} is the logic of totally preordered frames; see, e.g., {\rm \cite{book-modal}}. We have already mentioned its close relationship to \LAEC. We can actually say that \LAEC{} is stronger than {\rm S4.3}. Indeed, let us identify the basic expression of \LAEC{} not containing $\larger$ (or alternatively, not containing $\smaller$), with the {\rm S4.3} formulas. We claim that, for each {\rm S4.3} tautology $\alpha$, $\;\true \aentails{1} \alpha$ is provable in \LAEC.

Clearly, for any tautology $\alpha$ of \CPL, \LAEC{} proves $\true \aentails{1} \alpha$ by {\rm (A1')}. Furthermore, in \LAEC, $\true \aentails{1} \alpha \impl \beta$ and $\alpha \aentails{1} \beta$ are mutually derivable; hence from $\true \aentails{1} \alpha$ and $\true \aentails{1} \alpha \impl \beta$, \LAEC{} proves $\true \aentails{1} \beta$ by {\rm (A9)}.

Moreover, {\rm S4.3} can be axiomatised by the modal axioms $K$, $T$, $4$, and $H$ as well as the rule of necessitation. The claim follows for axiom $K$ from Lemma \ref{lem:LAEC} and for the axioms $T$ and $4$ from {\rm (A12)} and {\rm (A13)}, respectively. Furthermore, assume that \LAEC{} proves $\true \aentails{1} \alpha$. It then follows $\lnot\alpha \aentails{1} \false$, hence, by {\rm (A17)}, $\smaller\lnot\alpha \aentails{1} \smaller\false$, and, by {\rm (A14)}, $\smaller\lnot\alpha \aentails{1} \false$. We conclude $\true \aentails{1} \lnot\smaller\lnot\alpha$, that is, \LAEC{} emulates the rule of necessitation.

It remains to consider the axiom $H$, which we can express using the $\smaller$ operator in the following way:
\[ \lnot\smaller(\phi \land \lnot\smaller\psi) \lor \lnot\smaller(\psi \land \lnot\smaller\phi). \]
There is the following corresponding \LAEC{} proof. Assume $\smaller\phi \aentails{1} \smaller\psi$. Then $\phi \aentails{1} \smaller\psi$ by {\rm (A12)} and  $\phi \land \lnot\smaller\psi \aentails{1} \false$ by {\rm (A2)}. We conclude $\smaller(\phi \land \lnot\smaller\psi) \aentails{1} \false$ by {\rm (A17)} and {\rm (A14)} and hence $\true \aentails{1} \lnot\smaller(\phi \land \lnot\smaller\psi)$. Thus \LAEC{} proves $(\smaller\phi \aentails{1} \smaller\psi) \impl (\true \aentails{1} \lnot\smaller(\phi \land \lnot\smaller\psi))$. Similarly, $(\smaller\psi \aentails{1} \smaller\phi) \impl (\true \aentails{1} \lnot\smaller(\psi \land \lnot\smaller\phi))$ and thus, by {\rm (A15)}, $(\true \aentails{1} \lnot\smaller(\phi \land \lnot\smaller\psi)) \lor  (\true \aentails{1} \lnot\smaller(\psi \land \lnot\smaller\phi))$. Hence we finally have $\true \aentails{1} \lnot\smaller(\phi \land \lnot\smaller\psi) \lor \lnot\smaller(\psi \land \lnot\smaller\phi))$, as desired.
\end{remark}

\section{Approximate Entailment on Products of Chains}
\label{sec:LAEPC}

The logic \LAEC, which we have presented in the previous section, is designed for the case that the universe of discourse is endowed with a total order and that this order is, in a natural sense, compatible with the similarity relation. \LAEC{} is meant as a preparation of what we actually have in mind. In the present section, we consider the case that is more likely to occur in practice; we assume that the universe of discourse arises from the distinction with respect to several parameters, each of which refers to a totally ordered structure. That is, we will assume that the set of worlds is, not a chain but, a product of chains. We will use variables that refer to only one of these total orders; accordingly, we will deal with variables of different sorts.

We define the Logic for Approximate Entailment on Products of Chains, or \LAEPC{} for short, as follows. This time, our set of variables will be partitioned as follows. We fix a number $M \geq 1$ and for each $i = 1, \ldots, M$, we fix a finite number of variables $\phi_{i1}, \ldots, \phi_{iN_i}$. We will say that the variable $\phi_{ij}$ belongs to the {\it sort $i$}. Moreover, the truth constants $\false$ and $\true$ are considered to belong to every sort. Finally, we use an additional finite set $\psi_1, \ldots, \psi_N$ of variables that are not bound to be of a particular sort. We call the former variables {\it sorted} and the latter {\it unsorted}.\footnote{Actually, we will use the unsorted variables as a sort of syntactic sugar, in the sense that they will not play an active role in determining the possible worlds, but will be governed by the (truth-value of the) sorted ones. This will be reflected by axiom $(\textrm{A22})$ below.}

Basic expressions of \LAEPC{} as well as graded implications and formulas are defined as in case of \LAEC. We will keep using $\mathcal B$ to denote the set of basic expressions of \LAEPC. A basic expression $\phi$ all of whose variables belong to the sort $i$ will itself be said to belong to the sort $i$. We also say in this case that $\phi$ is one-sorted.

By a m.e.c. we mean now a conjunction of literals in which all sorted variables occur exactly once. Moreover, by a one-sorted m.e.c., we shall mean a conjunction of literals in which all the variables of a sort $i$ occur exactly once. Finally, we say that two basic expressions $\phi$ and $\psi$ are {\it disjoint-sorted} if all variables occurring in $\phi$ and $\psi$ are sorted and no variable occurring in $\phi$ is of the same sort as a variable occurring in $\psi$.

The models for the new logic are defined in the following way.

\begin{definition} \label{def:product-similarity-space}
For each $i = 1, \ldots, M$, let $(W_i, S_i, \worldleq_i)$ be a totally ordered similarity space based on $(V, \tand)$. Let $W = \prod_i W_i$. For each $v = (v_1, \ldots, v_M)$ and $w = (w_1, \ldots, w_M)$ in $W$, let
\begin{equation} \label{fml:total-similarity}
S(v, w) \;=\;
\min_{1 \leq i \leq M} \; S_i(v_i, w_i).
\end{equation}
Furthermore, for each $i$, we define the preorder $\below_i$ on $W$ as follows:
\begin{equation} \label{fml:preorders}
(v_1, \ldots, v_M) \below_i (w_1, \ldots, w_M) \quad\text{if}\quad v_i  \worldleq_i  w_i.
\end{equation}
Then $(W, S, (\below_i)_{i = 1, \ldots, M})$ is called a {\it component-wise ordered similarity space}.
\end{definition}

It is easily verified that, if $(W, S, (\below_i)_{i = 1, \ldots, M})$ is a component-wise ordered similarity space, then $(W,S)$ is a finite similarity space.

We will interpret the modal operations of \LAEC{} in the following way. Let the compo\-nent-wise ordered similarity space $(W, S, (\below_i)_{i = 1, \ldots, M})$ be given. For a subset $A \subseteq W$, we define
\begin{equation} \label{fml:diamond-on-similarity-space-multidimensional}
\begin{split}
\smaller A \;=\; & \{ w \in W \colon
 \text{for each $i$, there is a $v \in A$ such that $w \below_i v$} \}, \\
\larger A \;=\; & \{ w \in W \colon
 \text{for each $i$, there is a $v \in A$ such that $w \above_i v$} \}.
\end{split}
\end{equation}
Let $\larger_i$ and $\smaller_i$, for each $i = 1, \ldots, M$, be the modal operations for the totally ordered similarity space $(W_i, S_i,  \worldleq_i)$, defined according to (\ref{fml:diamond-on-similarity-space}). Then the modal operations on $W$, $\smaller$ and $\larger$, are determined by the operations $\larger_i$ and $\smaller_i$ in the following sense. Let us write, for some $C \subseteq W_i$,
\[ \pi C \;=\; \{ (w_1, \ldots, w_M) \in W \colon w_i \in C \}. \]
That is, $\pi C$ may be viewed as a cylindrical extension of $C$ into $W$. $\smaller$ applied to a set of this form is determined by $\smaller_i$. Indeed, for $C \subseteq W_i$, we have
\[ \smaller{\pi C} \;=\; \pi (\smaller_i C). \]
Furthermore, let $A$ be an arbitrary non-empty subset of $W$. Then
\begin{equation} \label{fml:smaller-1}
\smaller A \;=\; \bigcap_i \{ \pi(\smaller_i \{w_i\}) \colon w_i \in W_i \text{ such that } A \subseteq \pi(\smaller_i \{w_i\}) \}.
\end{equation}
To improve the expression on the right side, for each $i$, let $x_i$ be w.r.t.\ $\worldleq_i$ the largest element of $W_i$ such that $\pi \{x_i\}$ has a non-empty intersection with $A$. Then $\smaller A \subseteq \pi \smaller{\{x_i\}}$ and in fact we have
\begin{equation} \label{fml:smaller-2}
\smaller A \;=\; \bigcap_i \pi (\smaller_i{\{x_i\}}).
\end{equation}
Similar statements hold for $\larger$. An illustration can be found in Fig.\ \ref{fig:Modal-operator}.

\begin{figure}[h!]
\centering

\includegraphics[width=0.6\textwidth]{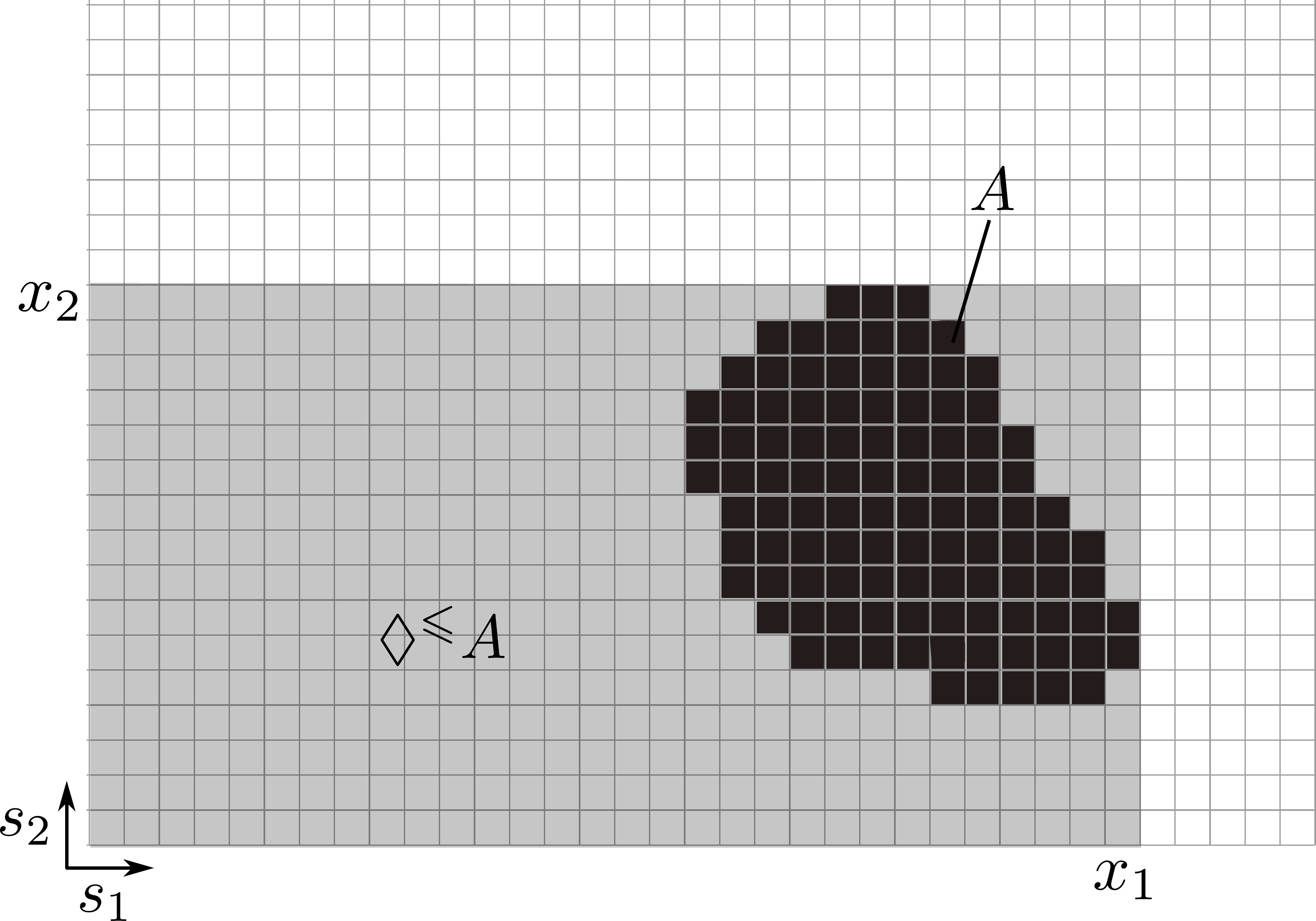}

\caption{Illustration of the meaning of the modal operator $\Diamond^\leqslant$.}
\label{fig:Modal-operator}
\end{figure}

Based on Definition \ref{def:product-similarity-space}, let us now define the notion of satisfaction for \LAEPC. 

\begin{definition} \label{def:semantics-LAEPC}
Let $(W, S, (\below_i)_{i = 1, \ldots, M})$ be a component-wise ordered similarity space. An evaluation for \LAEPC\ in $(W, S, (\below_i)_{i = 1, \ldots, M})$ is a mapping $e \colon {\mathcal B} \to {\mathcal P}(W)$, subject to the same conditions as in case of \LAE{} and to the following two additional  conditions: 
\begin{itemize}
\item[(i)] for any variable $\phi$ of sort $i$, $e(\phi) = \pi A$ for some $A \subseteq W_i$;
\item[(ii)] for any unsorted variable $\alpha$, $e(\alpha)$ is a union of intersections of sets of the form $e(\phi)$ or $W \setminus e(\phi)$, where $\phi$ is a variable of any sort;
\item[(iii)] for an arbitrary basic expression $\phi$, we require
\begin{align*}
e(\smaller\phi) \;=\; \smaller{e(\phi)}, \\
e(\larger\phi) \;=\; \larger{e(\phi)}.
\end{align*}
\end{itemize}
Moreover, the notions of {\it satisfaction} and {\it semantic entailment} are defined analogously to the case of \LAEC. We write $(W, S, (\below_i)_{i = 1, \ldots, M}, e) \models \Phi$ to denote that an evaluation $e$ in $(W, S, (\below_i)_{i = 1, \ldots, M})$ satisfies an \LAEPC\ formula $\Phi$, and $\mathcal{T} \modelsLAEPC \Phi$ to denote that a theory $\cal T$ semantically entails $\Phi$ in \LAEPC.
\end{definition}

We see that Definition \ref{def:semantics-LAEPC} generalises Definition \ref{def:semantics-LAEC}: when restricting the variables to a single sort $i$, the semantic entailment in \LAEPC{} resembles the case of \LAEC. Note furthermore the effect of part (ii) of Definition \ref{def:semantics-LAEPC}. An unsorted variable $\alpha$ is, by definition, not bound to a particular sort. It is, however, required that $\alpha$ is interpreted in the same way as some Boolean combination of sorted variables. Thus $\alpha$ is interpreted like a disjunction of m.e.c.s, that is, by a disjunction of conjunctions of sorted variables or their negation.

Again, we compile some basic properties of the operators (\ref{fml:diamond-on-similarity-space-multidimensional}). With reference to the notation of (\ref{fml:diamond-on-similarity-space-multidimensional}), we will write for any $A \subseteq W_{i_1} \times \ldots \times W_{i_k}$, where $1 \leq i_1 < \ldots < i_k \leq M$,
\[ \pi A \;=\; \{ (w_1, \ldots, w_M) \in W \colon (w_{i_1}, \ldots, w_{i_k}) \in A \}. \]

\begin{lemma} \label{lem:diamond-on-similarity-space-multidimensional}
Let $(W, S, (\below_i)_{i = 1, \ldots, M})$ be a component-wise ordered similarity space. Then, for any $A, B \subseteq W$ and $c \in V$, the following holds:
\EnumeratedList
\Number{i} $A \subseteq \smaller A$.
\Number{ii} $\smaller{\smaller A} = \smaller A$.
\Number{iii} $\smaller\emptyset = \emptyset$.
\Number{iv} If $A \subseteq U_c(B)$, then $\smaller A \subseteq U_c \smaller B$.

In particular, $A \subseteq B$ implies $\smaller A \subseteq \smaller B$.
\EndofEnumeratedList
In addition, each statement {\rm (i)}--{\rm (iv)} still holds when we replace all symbols ``$\smaller$'' by ``$\larger$''.

Moreover, let $1 \leq i_1 < \ldots < i_k \leq M$ and $1 \leq j_1 < \ldots < j_l \leq M$ such that $\{i_1, \ldots, i_k\}$ and $\{j_1, \ldots, j_l\}$ are disjoint. Then we have:
\EnumeratedList
\Number{v} Let $A, C \subseteq W_{i_1} \times \ldots \times W_{i_k}$ and $B, D \subseteq W_{j_1} \times \ldots \times W_{j_l}$. Assume that $A, B, C, D$ are non-empty. Then $\pi A \subseteq U_c(\pi C)$ and $\pi B \subseteq U_c(\pi D)$ if and only if $\pi(A \times B) \subseteq U_c(\pi(C \times D))$.
\Number{vi} Let $A \subseteq W$ and let $B \subseteq W_{i_1} \times \ldots \times W_{i_k}$ and $C \subseteq W_{j_1} \times \ldots \times W_{j_l}$. Then $\smaller A \cap \pi B \cap \pi C = \emptyset$ if and only if $\smaller A \cap \pi B = \emptyset$ or $\smaller A \cap \pi C = \emptyset$. A similar statement holds for ``$\larger$'' replacing ``$\smaller$''.
\EndofEnumeratedList
Finally, we have:
\EnumeratedList
\Number{vii} Let $A \subseteq W_i$ for some $i$ and let $B \subseteq W$. If $\pi A \cap \smaller B = \emptyset$, then $\larger(\pi A) \cap \smaller B = \emptyset$. Similarly, if $\pi A \cap \larger B = \emptyset$, then $\smaller(\pi A) \cap \larger B = \emptyset$.
\EndofEnumeratedList
\end{lemma}

\begin{proof}
(i) and (iii) follow easily from the definition (\ref{fml:diamond-on-similarity-space-multidimensional}).

(ii) follows from (\ref{fml:smaller-2}).

(iv) Let $A, B \subseteq W$ such that $A \subseteq U_c(B)$. We have to show that $\smaller A \subseteq U_c(\smaller B)$.

For each $i = 1, \ldots, M$, let $a_{ii} \in W_i$ be the largest element w.r.t.\ $\worldleq_i$ such that there is an $a_i = (a_{i1}, \ldots, a_{ii}, \ldots, a_{iM}) \in A$; cf.\ (\ref{fml:smaller-2}). Then $\smaller A = \{ (w_1, \ldots, w_M) \colon w_i \worldleq_i a_{ii} \text{ for all $i$} \}$. Furthermore, by assumption, there is, for each $i$, a $b_i = (b_{i1}, \ldots, b_{iM}) \in B$ such that $S(a_i, b_i) \geq c$. Then $(b_{11}, \ldots, b_{MM}) \in \smaller B$. Furthermore, because $S_i(a_{ii}, b_{ii}) \geq c$ for each $i$, we have $S((a_{11}, \ldots, a_{MM}), (b_{11}, \ldots, b_{MM})) \geq c$ and hence $(a_{11}, \ldots, a_{MM}) \in U_c(\smaller B)$. The claim follows.

(v) Note first that $\pi A \subseteq U_c(\pi C)$ holds if and only if, for any $(v_{i_1}, \ldots, v_{i_k}) \in A$ there is a $(w_{i_1}, \ldots, w_{i_k}) \in C$ such that $S_{i_p}(v_{i_p}, w_{i_p}) \geq c$ for all $p = 1, \ldots, k$.

We conclude that $\pi A \subseteq U_c(\pi C)$ and $\pi B \subseteq U_c(\pi D)$ if and only if, for any $(v_{i_1}, \ldots,$ $v_{i_k}) \in A$ and $(v_{j_1}, \ldots, u_{j_l}) \in B$, there is a $(w_{i_1}, \ldots, w_{i_k}) \in C$ and a $(w_{j_1}, \ldots, w_{k_l}) \in D$ such that  $S_{i_p}(v_{i_p}, w_{i_p}) \geq c$ for all $p = 1, \ldots, k$ and $S_{j_q}(v_{j_q}, w_{j_q}) \geq c$ for all $q = 1, \ldots, l$. Provided that the considered sets are all non-empty, the latter statement is obviously equivalent to $\pi(A \times B) \subseteq U_c(\pi(C \times D))$.

(vi) From $A \subseteq W$, let $e_1 \in W_1, \ldots, e_M \in W_M$ be defined like in (\ref{fml:smaller-2}). Then $\smaller A \cap \pi B \neq \emptyset$ holds if and only if there is a $(v_{i_1}, \ldots, v_{i_k}) \in B$ such that $e_{i_1} \worldleq_{i_1} v_{i_1}, \ldots, e_{i_k} \worldleq_{i_k} v_{i_k}$. Note furthermore that $\pi B \cap \pi C = \pi(B \times C)$.

We conclude that $\smaller A \cap \pi B \cap \pi C \neq \emptyset$ if and only if $\smaller A \cap \pi B \neq \emptyset$ and $\smaller A \cap \pi C \neq \emptyset$. This proves the indicated statement; the version with ``$\larger$'' replacing ``$\smaller$'' is seen similarly.

(vii) The assertion is trivial if $B = \emptyset$. Otherwise, $\pi A \cap \smaller B = \emptyset$ iff $a \in A$ implies that $a >_i b$ for all $b \in B$. The first part follows; the second one is seen analogously.
\end{proof}

Note that, as in the case of \LAEC, in the present case of \LAEPC\ we will have the axiom (A19) that will allow us to combine conclusions in a conjunctive way. Its soundness is due to the following lemma.

Here, by an {\it orthotope} we mean a Cartesian product of intervals. Note that, for any non-empty subset $A$ of $W$, $\smaller A \cap \larger A$ is the smallest orthotope containing $A$.

\begin{lemma} \label{lem:intersection-of-orthotopes}
Let $(W, S, (\below_i)_{i = 1, \ldots, M})$ be a component-wise ordered similarity space. Let $A, B \subseteq W$ be orthotopes with a non-empty intersection and let $c \in V$. Then
\[ U_c(A \cap B) \;=\; U_c(A) \cap U_c(B). \]
\end{lemma}

\begin{proof}
This is an easy consequence of Lemma \ref{lem:intersection-of-intervals}.
\end{proof}

\begin{figure}[h!]
\centering

\includegraphics[width=0.65\textwidth]{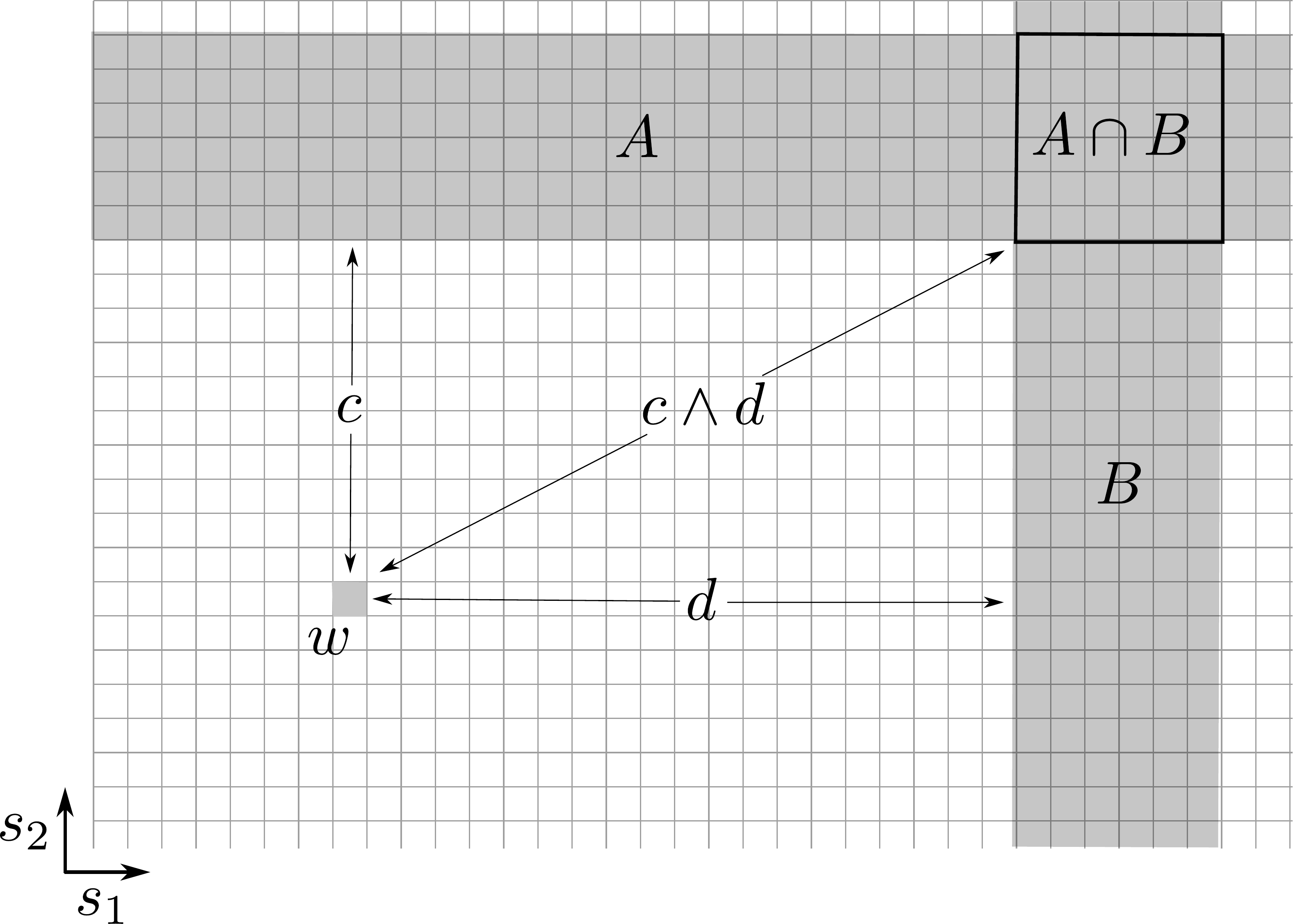}

\caption{Illustration of Lemma \ref{lem:intersection-of-orthotopes}.}
\label{fig:Distance-from-intersection}
\end{figure}

We proceed by proposing an axiomatisation of \LAEPC.

\begin{definition} \label{def:LAEPC-proofs}
The axioms of \LAEPC{} are  the following ones, for any basic expressions of \LAEPC{} $\phi, \psi, \chi \in {\mathcal B}$ and any $c, d \in V$:

axioms (A1') and (A2)--(A11);

axioms (A12)--(A14), (A17), (A19);

axiom (A15), where $\phi$ and $\psi$ belong to one coinciding sort;

axiom (A16), where $\epsilon$ is a one-sorted m.e.c.;

axiom (A18), where $\phi$ is one-sorted;

as well as
\BeginAxioms
\Axiom{A20} $\lnot(\phi \land \phi' \aentails{1} \false) \,\impl\, ((\phi \aentails{c} \psi) \land (\phi' \aentails{c} \psi') \eq (\phi \land \phi' \aentails{c} \psi \land \psi'))$, \\ where $\phi \land \psi$ and $\phi' \land \psi'$ are disjoint-sorted
\Axiom{A21} $((\larger\phi) \land \chi \land \psi \aentails{1} \false) \impl ((\larger\phi) \land \chi \aentails{1} \false) \lor ((\larger\phi) \land \psi \aentails{1} \false)$ and $((\smaller\phi) \land \chi \land \psi \aentails{1} \false) \impl ((\smaller\phi) \land \chi \aentails{1} \false) \lor ((\smaller\phi) \land \psi \aentails{1} \false)$, \\
where $\chi$ and $\psi$ are disjoint-sorted
\Axiom{A22} $(\epsilon \aentails{1} \alpha) \lor (\epsilon \aentails{1} \lnot\alpha)$, \\ where $\epsilon$ is a m.e.c.
\EndAxioms
Moreover, the only rule of \LAEPC{} is (MP).
\end{definition}

\begin{theorem} \label{thm:completeness-LAEPC}
Let $\mathcal T$ be a theory and $\Phi$ be a formula of \LAEPC. Then ${\mathcal T} \provesLAEPC \Phi$ if and only if ${\mathcal T} \modelsLAEPC \Phi$.
\end{theorem}

\begin{proof}
We check again the ``only if'' part first. The soundness of (A1)--(A10) follows from the soundness part of Theorem \ref{thm:completeness-LAE}. The soundness of (A15)--(A16) follows from part (i) of Definition \ref{def:semantics-LAEPC} and the soundness part of Theorem \ref{thm:completeness-LAEC}. Moreover, the soundness of (A22) follows from part (ii) of Definition \ref{def:semantics-LAEPC}.

The soundness of (A12), (A13), (A14), (A17), and (A18) holds by parts (i), (ii), (iii), (iv), and (vii) of Lemma \ref{lem:diamond-on-similarity-space-multidimensional}, respectively. (A19) is sound by Lemma \ref{lem:intersection-of-orthotopes}. Finally, (A20) and (A21) are sound by part (v) and (vi) of Lemma \ref{lem:diamond-on-similarity-space-multidimensional}, respectively.

To see the ``if'' part, assume that $\mathcal T$ does not prove $\Phi$. Extending $\mathcal T$ if necessary, we can again assume that $\mathcal T$ is complete.

On the set of basic expressions $\mathcal B$, we define the relation $\stronger$, its symmetrisation $\equallystrong$, and the Boolean algebra $\mathcal L$ of $\equallystrong$-classes like in the proof of Theorem \ref{thm:completeness-LAE}. By (A17), also $\smaller$ and $\larger$ are compatible with $\equallystrong$ and hence induce unary operations on $\mathcal L$. Note furthermore that by (A22), each unsorted variable $\alpha$ is equivalent to a basic expression in which only sorted variables occur.

Let us consider a sort $1 \leq i \leq M$. Let ${\mathcal B}_i \subseteq {\mathcal B}$ consist of all basic expressions of the sort $i$. Moreover, let ${\mathcal L}_i$ be the Boolean subalgebra of $\mathcal L$ generated by the $\equallystrong$-classes of elements of ${\mathcal B}_i$. Note that all axioms of \LAEC{} apply to ${\mathcal B}_i$. Hence we can proceed like in the proof of Theorem \ref{thm:completeness-LAEC} to construct a totally ordered similarity space $(W_i, S_i, \worldleq_i)$ such that $W_i = \{ \equallystrongcl{\epsilon} \in {\mathcal L}_i \colon \text{$\epsilon$ is a m.e.c.\ of sort $i$ such that ${\mathcal T} \doesnotprove \epsilon \aentails{1} \false$} \}$ and, for any $\phi \in {\mathcal B}_i$, $\smaller \phi$ is the supremum of all m.e.c.s $\epsilon$ of sort $i$ such that $\equallystrongcl{\epsilon} \worldleq_i \equallystrongcl{\delta}$ for some m.e.c.\ $\delta \stronger \phi$ of sort $i$, and similarly for $\larger$.

Let now $W$ be the set consisting of the $\equallystrongcl{\epsilon}$'s, where $\epsilon$ is a m.e.c.\ such that ${\mathcal T} \doesnotprove \epsilon \aentails{1} \false$. Let $\equallystrongcl{\epsilon_i} \in W_i$ for $i = 1, \ldots, M$; then ${\mathcal T} \doesnotprove \epsilon_i \aentails{1} \false$ for any $i$ and it follows from (A21) that ${\mathcal T} \doesnotprove \bigwedge_i \epsilon_i \aentails{1} \false$ and consequently $\equallystrongcl{\bigwedge_i \epsilon_i} \in W$. We conclude that $W$ can be identified with the direct product of the $W_i$'s. Under this identification, we extend $\worldleq_i$ to a preorder $\leq_i$ on $W$ according to (\ref{fml:preorders}).

For $\equallystrongcl{\delta}, \equallystrongcl{\epsilon} \in W$, we define $S(\equallystrongcl{\delta}, \equallystrongcl{\epsilon})$ by (\ref{fml:derived-similarity}). By (A20), we conclude that $S$ depends on the $S_i$ according to (\ref{fml:total-similarity}). Hence $(W, S, (\below_i)_{i = 1, \ldots, M})$ is a component-wise ordered similarity space.

By (A14), we have $\smaller\equallystrongcl{\false} = \larger\equallystrongcl{\false} = \equallystrongcl{\false}$. Our next aim is to show that, for any $\equallystrongcl{\phi} \in {\mathcal L} \backslash \{\equallystrongcl{\false}\}$,
\begin{equation} \label{fml:LAEPC-proof}
\begin{split}
\smaller\equallystrongcl{\phi} \;=\; \bigwedge \{ \smaller\equallystrongcl{\epsilon} \colon
   \text{$\epsilon$ is a one-sorted m.e.c.\ such that
   $\phi \stronger \smaller\epsilon$} \}, \\
\larger\equallystrongcl{\phi} \;=\; \bigwedge \{ \larger\equallystrongcl{\epsilon} \colon
   \text{$\epsilon$ is a one-sorted m.e.c.\ such that
   $\phi \stronger \larger\epsilon$} \}.
\end{split}
\end{equation}
It will then follow that $\smaller$ and $\larger$ are defined on $\mathcal L$ in accordance with (\ref{fml:smaller-1}).

We restrict to the first part of (\ref{fml:LAEPC-proof}); the second part is shown analogously. The ``$\stronger$'' relation follows from (A17a), (A13a), and (A19). To see the ``$\lessstrong$'' relation, let $\delta$ be a m.e.c.\ such that $\delta \land \smaller\phi \equallystrong \false$. Let $\delta \equallystrong \epsilon_1 \land \ldots \land \epsilon_M$, where $\epsilon_i$ is, for each $i$, a one-sorted m.e.c.\ belonging to the sort $i$. By (A21), there is an $i$ such that $\epsilon_i \land \smaller\phi \equallystrong \false$. By (A18), it further follows $\larger\epsilon_i \land \smaller\phi \equallystrong \false$. Hence $\phi \stronger \smaller\phi \stronger \lnot\larger\epsilon_i \equallystrong \smaller\epsilon_i'$, where $\epsilon_i'$ is the predecessor of $\epsilon_i$ w.r.t.\ $\worldleq_i$. Moreover, $\delta \land \smaller\epsilon_i' \stronger \epsilon_i \land \smaller\epsilon_i' \equallystrong \false$, and the claim follows.

In particular, $W$ is the set of atoms of $\mathcal L$. We conclude that the mapping $e \colon {\mathcal B} \to {\mathcal P}(W)$ defined by (\ref{fml:evaluation}), that is,
\[ e(\phi) \;=\;
   \{ \equallystrongcl{\epsilon} \in W \colon \mathcal{T} \proves \epsilon \aentails{1} \phi \}, \]
is an evaluation in $(W, S, (\below_i)_{i = 1, \ldots, M})$. We argue as in the proof of Theorem \ref{thm:completeness-LAE} to see that (\ref{fml:LAE-proof}) holds. Hence the evaluation $e$ satisfies all elements of $\mathcal T$ but not $\Phi$.
\end{proof}

The system \LAEPC{} is thus an advanced variant of the original logic \LAEC{} that, in a scenario with different sorts (or attributes) and a global similarity relation built from individual ones on each sort, is able to cope with a restricted conjunctive inference pattern.  

At this point it is interesting to go back to the example introduced in Section \ref{sec:introduction}. Remember we had three propositions $\alpha$, $\beta$, $\gamma$ denoting the following properties of a car:

\begin{tabular}{ll}
$\alpha$ & ``power(car) $=$ 110 CV'' \\
$\beta$  & ``price(car) $\geq$  20 000 \euro'' \\
$\gamma$ & ``consumption(car) $\geq$ 6 L/100km''
\end{tabular}

\noindent 
and that our domain knowledge was modelled by a theory $\mathcal T$ containing the graded implications $\alpha  \aentails{c} \beta,  \alpha  \aentails{d} \gamma$. Requiring in addition that neither $\beta$ nor $\gamma$ is contradictory, let our theory be such that
\[ \mathcal T \supseteq \{\alpha  \aentails{c} \beta, \;\; \alpha  \aentails{d} \gamma, \;\; \lnot(\beta \aentails{1} \false), \;\; \lnot(\gamma \aentails{1} \false) \}.\]

\noindent
Further assume, as suggested already in Section \ref{sec:introduction}, that $\beta$ and $\gamma$ belong to different sorts (power and price, respectively) and that $\alpha$ is unsorted. Moreover, as the intended semantics of $\beta$ and $\gamma$ is to denote upwards closed intervals in the range of prices and consumption, we can assume that they are of the form $\larger \beta'$ and $\larger \gamma'$, respectively. By (A21), \LAEPC{} proves from $\mathcal T$ that $\lnot(\beta \land \gamma \aentails{1} \false)$ and hence we can apply (A19). Thus \LAEPC{} allows us to approximately conclude from $\mathcal T$ to the degree $\min(c, d)$ that if the power of a car is 110 CV then its price is above 20 000 \euro\ {\em and} its consumption will be at least 6 litres per 100 km, that is,
\[ \mathcal T \provesLAEPC \alpha  \aentails{\min(c,d)} \beta \land \gamma.\]

\section{Conclusions}
\label{sec:conclusion}

In this paper we have been concerned with extending the logic \LAE{}, a logic for reasoning about graded similarity-based approximate conditionals $\varphi \aentails{c} \psi$, to allow for a conjunctive closure of the conclusions of these conditionals, a feature that is lacking in its original formulation. The semantics of these logical systems is based on Kripke-like structures, that we have called similarity spaces, consisting of a set of worlds equipped with a fuzzy similarity relation. For our purposes, we have considered  two particular classes of these structures. In a first step, we have considered similarity spaces where the set of worlds is endowed with a total order and the similarity relation is compatible with it. Under these assumptions we have shown that, in the resulting logic \LAEC, one can derive the conditional $\alpha  \aentails{c} \beta \land \gamma$ from $\alpha  \aentails{c} \beta$ and $\alpha  \aentails{c} \gamma$, as soon as $\beta$ and $\gamma$ are non-contradictory propositions interpreted as intervals in the chain of worlds. The possibility to refer to intervals can be achieved by means of introducing two modal operators $\Diamond^{\leq}$ and $\Diamond^{\geq}$ into the basic language, whose semantics is given by the total order  in our enriched similarity spaces taken as an accessibility relation. Then, in a second step, we have generalised this approach to axiomatise a many-sorted logic \LAEPC{} whose corresponding classes of fuzzy similarity structures are Cartesian products of totally ordered similarity spaces, each one for a different sort of the language. In this case , the operators $\Diamond^{\leq}$ and $\Diamond^{\geq}$ are able to capture properties of Cartesian products of intervals.  

A number of open issues remain to be addressed in future developments. For instance, even if the logic \LAEPC{} already has a much better expressive power than the original \LAEC{}, it would certainly be desirable to have a logic without the technical constraints that we introduced to validate the conjunctive combination of conclusions in graded implications.  Another interesting extension to study is to allow for a more general language, already starting with \LAE{}, where the $\aentails{c}$ operators can be nested. Indeed, one could express the operator  $\aentails{c}$ in terms of a graded possibility $KTB$-modality $\Diamond_c$ (like in \cite{EGGR}) together with a global S5 necessity modality $\Box$, namely to express $\varphi \aentails{c} \psi$ as $\Box(\varphi \to \Diamond_c \psi)$.\footnote{We are grateful to an anonymous reviewer for pointing out this question.} Yet another  alternative approach to explore is to introduce a notion of context into these graded implications, where contexts basically encode subsets of possible worlds that enforce the validity of the graded implications that they are qualifying. This approach was already considered in \cite{DPEGG} in the setting of graded consequence relations. A further important question to be addressed is the complexity of these logics and efficient proof methods for them; cf.\ \cite{AlOl,AOP}.

\subsection*{Acknowledgements.} The authors sincerely thank the anonymous reviewers for their useful and constructive comments that have significantly helped to improve the paper. Vetterlein acknowledges the support of the Austrian Science Fund (FWF): project I 1923-N25 (New perspectives on residuated posets). Esteva and Godo acknowledge the support of the Spanish MINECO project EdeTRI TIN2012-39348-C02-01 and of the Catalan Government grant 2014 SGR 118.

\end{document}